\documentclass[11pt]{amsart}
\usepackage[margin=1.35in]{geometry}
\usepackage{amscd,amsmath,amsxtra,amsthm,amssymb,stmaryrd,xr,mathrsfs,mathtools,enumerate,commath, comment}
\usepackage{stmaryrd}
\usepackage{xcolor}
\usepackage{commath}
\usepackage{comment}
\usepackage{tikz-cd}
\usepackage{longtable} 
\usepackage{pdflscape} 
\usepackage{booktabs}
\usepackage{hyperref}
\definecolor{vegasgold}{rgb}{0.77, 0.7, 0.35}
\definecolor{darkgoldenrod}{rgb}{0.72, 0.53, 0.04}
\definecolor{gold(metallic)}{rgb}{0.83, 0.69, 0.22}
\hypersetup{
 colorlinks=true,
 linkcolor=darkgoldenrod,
 filecolor=brown,      
 urlcolor=gold(metallic),
 citecolor=darkgoldenrod,
 }

\usepackage[all,cmtip]{xy}

\DeclareFontFamily{U}{wncy}{}
\DeclareFontShape{U}{wncy}{m}{n}{<->wncyr10}{}
\DeclareSymbolFont{mcy}{U}{wncy}{m}{n}
\DeclareMathSymbol{\Sh}{\mathord}{mcy}{"58}
\usepackage[T2A,T1]{fontenc}
\usepackage[OT2,T1]{fontenc}

\newtheorem{theorem}{Theorem}[section]

\newtheorem*{ass*}{Assumption}
\newtheorem{definition}[theorem]{Definition}

\newtheorem{proposition}[theorem]{Proposition}

\newtheorem{lthm}{Theorem}

\newcommand{\cB}{\mathcal{B}}

\newcommand{\cF}{\mathcal{F}}
\newcommand{\cD}{\mathcal{D}}
\newcommand{\cM}{\mathcal{M}}

\newcommand{\cE}{\mathcal{E}}

\newcommand{\Z}{\mathbb{Z}}

\newcommand{\Q}{\mathbb{Q}}
\newcommand{\F}{\mathbb{F}}

\newcommand{\cC}{\mathcal{C}}

\newcommand{\cS}{\mathcal{S}}

\newcommand{\cMl}{\mathcal{M}(\op{GL}_2(\F_\ell))}
\newcommand{\cMX}{\mathcal{M}(\op{GL}_2(\F_\ell); X)}
\newcommand{\cMY}[1]{\mathcal{M}(\op{GL}_2(\F_\ell); #1)}

\newcommand{\cFX}{\mathcal{F}(\op{GL}_2(\F_\ell); X)}
\newcommand{\cFY}[1]{\mathcal{F}(\op{GL}_2(\F_\ell); #1)}

\newcommand{\op}[1]{\operatorname{#1}}

\numberwithin{equation}{section}

\begin{document}

\title[Number of number fields with Galois group $\op{GL}_2(\F_\ell)$]{Lower bounds for the number of number fields with Galois group $\op{GL}_2(\F_\ell)$}

\author[A.~Ray]{Anwesh Ray}
\address[Ray]{Chennai Mathematical Institute, H1, SIPCOT IT Park, Kelambakkam, Siruseri, Tamil Nadu 603103, India}
\email{ar2222@cornell.edu}


\maketitle

\begin{abstract}
  Let $\ell\geq 5$ be a prime number and $\F_\ell$ denote the finite field with $\ell$ elements. We show that the number of Galois extensions of the rationals with Galois group isomorphic to $\op{GL}_2(\F_\ell)$ and absolute discriminant bounded above by $X$ is asymptotically at least $\frac{X^{\frac{\ell}{12(\ell-1)\#\op{GL}_2(\F_\ell)}}}{\log X}$. We also obtain a similar result for the number of surjective homomorphisms $\rho:\op{Gal}(\bar{\Q}/\Q)\rightarrow \op{GL}_2(\F_\ell)$ ordered by the prime to $\ell$ part of the Artin conductor of $\rho$. 
\end{abstract}

\section{Introduction}
\par Let $G$ be a transitive subgroup of $S_n$. Given a number field $K$ with $[K:\Q]=n$, let $\tilde{K}$ denote its Galois closure over $\Q$. We view $\op{Gal}(\tilde{K}/\Q)$ as a transitive permutation subgroup of $S_n$, where the permutation action is defined by its action on the $n$ embeddings of $K$ into $\mathbb{C}$. For any real number $X>0$, let $N(G; X)$ be the number of number fields $K$ of degree $n$, with absolute discriminant at most $X$, such that $\op{Gal}(\tilde{K}/\Q)$ is isomorphic to $G$ as a permutation subgroup of $S_n$. In \cite{malle2002distribution}, Malle made a precise conjecture about the asymptotic growth of $N(G;X)$ as a function of $X$. The conjecture posits that
\[N(G;X)\sim X^{\frac{1}{a(G)}} (\log X)^{b(G)-1},\] where $a(G)$ and $b(G)$ are numbers which depend on the permutation representation of $G$. Malle's conjecture has been proven for the abelian groups \cite{maki1985density}, $S_n$ for $n\leq 5$ \cite{davenport1971density, bhargava2005density,bhargava2008density, bhargava2010density,belabas2010discriminants}, the dihedral group $D_4$ \cite{cohen2002enumerating}, and various groups of the form $A\times S_n$, where $A$ is a finite abelian group and $n\leq 5$, cf. \cite{wang2021malle}. This list is not exhaustive. Motivated by such developments, various asymptotic lower bounds for the groups $S_n$ and $A_n$ have been proven, cf. \cite{ellenberg2006number, pierce2020effective}. 
\subsection{Main result} A large class of finite groups of interest are the groups of the form $G(\F_q)$, where $G$ is an algebraic group over the finite field $\F_q$. These are groups that arise naturally in the study of Galois representations associated with abelian varieties, motives and automorphic forms. Let $\ell\geq 5$ be an odd prime number and let $\F_\ell$ be the finite field with $\ell$ elements. One of the simplest examples to consider are the groups $G=\op{GL}_2(\F_\ell)$. We view $\op{GL}_2(\F_\ell)$ as a subgroup of $S_{\#\op{GL}_2(\F_\ell)}$, taken with respect to the regular representation, and prove asymptotic lower bounds for the number of Galois number fields $L/\Q$ with Galois group $\op{Gal}(L/\Q)$ isomorphic to $\op{GL}_2(\F_\ell)$, and bounded absolute discriminant.
\par Given a number field $L$, set $\Delta_L$ to denote the discriminant of $L/\Q$. Set $\cM(\op{GL}_2(\F_\ell))$ to denote the set of all pairs $(L, \rho)$, where
\begin{itemize}
    \item $L/\Q$ is a Galois extension with Galois group $\op{Gal}(L/\Q)$ isomorphic to $\op{GL}_2(\F_\ell)$;
    \item $\rho:\op{Gal}(L/\Q)\xrightarrow{\sim} \op{GL}_2(\F_\ell)$ is an isomorphism which arises from the $\ell$-torsion of an elliptic curve $E$.
\end{itemize}
We take $N_\rho$ to denote the prime to $\ell$ part of the Artin conductor of $\rho:\op{G}_\Q\rightarrow \op{GL}_2(\F_\ell)$. This quantity is also known as the \emph{optimal level}, with reference to Serre's conjecture, cf. \cite[section 1.3]{ribet1999lectures}. Note that if two isomorphisms $\rho, \rho':\op{Gal}(L/\Q)\xrightarrow{\sim} \op{GL}_2(\F_\ell)$ differ by an inner automorphism of $\op{GL}_2(\F_\ell)$, then, $N_\rho=N_{\rho'}$. Taking $X\in \mathbb{R}_{>0}$, one is interested in obtaining asymptotic lower bounds for the cardinalities of the following sets
\[\begin{split} & \cMX:=\{(L, \rho)\in \cMl\mid N_\rho\leq X\},\\
& \cFX:=\{L\mid L/\Q\text{ is Galois with }\op{Gal}(L/\Q)\simeq \op{GL}_2(\F_\ell)\text{, and } |\Delta_L|\leq X\}.
\end{split}\]
We note that $\# \cF(\op{GL}_2(\F_\ell); X)$ is simply $N(\op{GL}_2(\F_\ell);X)$ with respect to the natural inclusion of $\op{GL}_2(\F_\ell)$ into $S_{\# \op{GL}_2(\F_\ell)}$. Given a set $S$, let $\#S\in \Z_{\geq 0}\cup \{\infty\}$ denote its cardinality. Given functions $f(X)$ and $g(X)$ taking on non-negative values, we set $f(X)\gg g(X)$ (resp. $f(X)=O(g(X))$) to mean that there is a constant $C>0$, independent of $X$, such that $f(X)\geq C g(X)$ (resp. $f(X)\leq C g(X)$). The main result is stated below.
\begin{lthm}\label{main thm}
    Let $\ell\geq 5$ be a prime number. Then,
    \[\begin{split}
    &\# \cMX\gg \frac{X^{\frac{1}{12}}}{\log X};\\
    & N(\op{GL}_2(\F_\ell);X)=\# \cFX\gg\frac{X^{\frac{\ell}{12(\ell-1)\#\op{GL}_2(\F_\ell)}}}{\log X},
    \end{split}\]
    where the implied constant depends only on $\ell$. 
\end{lthm}
Lipnowski \cite{lipnowski2021bhargava} proposed a heuristic for counting $\op{GL}_2(\F_\ell)$-extensions of $\Q$, that satisfy the same properties as the extensions that arise from modular Galois representations. In \emph{loc. cit.}, a homomorphism $\rho: \op{Gal}(\bar{\Q}/\Q)\rightarrow \op{GL}_2(\F_\ell)$ is said to be \emph{locally modular}, if it satisfies the same local conditions as that of a Galois representation associated with a Hecke eigencuspform of weight $2$. This conjecture is based on Bhargava's local-global heuristics for counting $G$-number fields \cite{bhargava2007mass}, and predicts that if $G(X)$ is the number of locally modular Galois representations with squarefree conductor $N\in [X, 2X]$, then, 
\[G(X)\sim X\log X \prod_v \left(1-3v^{-2}+2v^{-3}\right), \]where $v$ runs over all prime numbers. The representations constructed in this article arise from elliptic curves, and thus are locally modular. 
\par For any finite group $G$ with $n$ elements and inclusion $G\hookrightarrow S_n$ with respect to the regular representation, the value of $a(G)$ is $\left(\frac{p-1}{p}\right)\# G$, where $p$ is the smallest prime divisor of $G$ (cf. \cite[p. 318]{malle2002distribution}). Therefore, the power of $X$ in the predicted asymptotic for $N(\op{GL}_2(\F_\ell);X)$ is $\frac{2}{\#\op{GL}_2(\F_\ell)}$. In our main result for $N(\op{GL}_2(\F_\ell);X)$, the power of $X$ obtained is thus $\frac{1}{24}\left(\frac{\ell}{\ell-1}\right)$ times the conjectured value. In particular the conjecture predicts that $\frac{1}{a(\op{GL}_2(\F_\ell))}\sim \frac{2}{\ell^4}$ (as $\ell\rightarrow \infty$), and the power of $X$ in our bound is $\sim \frac{1}{12\ell^4}$. Thus, the correct power of $\ell$ has been achieved, and we are off by a factor of $24$.
\subsection{Method of proof} Let us briefly describe the methodology behind the proof of Theorem \ref{main thm}. Given an elliptic curve $E_{/\Q}$ and a prime number $\ell$, one considers the associated Galois representation 
\[\bar{\rho}_{E,\ell}:\op{Gal}(\bar{\Q}/\Q)\rightarrow \op{GL}_2(\F_\ell)\] on the $\ell$-torsion subgroup of $E(\bar{\Q})$. The $\ell$ division field $L_\ell(E)$ is the field extension of $\Q$, which is fixed by the kernel of $\bar{\rho}_{E, \ell}$. The number field $L_\ell(E)$ is Galois over $\Q$, and the Galois group $\op{Gal}(L_\ell(E)/\Q)$ is isomorphic to $\op{GL}_2(\F_\ell)$ precisely when $\bar{\rho}_{E, \ell}$ is surjective. We consider the family of all elliptic curves over $\Q$, ordered according to \emph{height}, and relate it to the family of $\op{GL}_2(\F_\ell)$-extensions of $\Q$, ordered by discriminant (or Artin conductor). It follows from a theorem of Duke \cite{duke1997elliptic} that for almost all elliptic curves $E_{/\Q}$, the residual representation $\bar{\rho}_{E, \ell}$ is surjective. Given a Galois extension $L/\Q$ with Galois group isomorphic to $\op{GL}_2(\F_\ell)$, we obtain an asymptotic upper bound for the number of elliptic curves $E_{/\Q}$ with bounded height, for which $L_\ell(E)=L$, cf. Proposition \ref{main prop}. This involves an application of a result of the author and Weston \cite[Proposition 3.2]{ray2023diophantine}. In fact, the argument can be traced back to a result of Duke \cite[Theorem 2]{duke1997elliptic}. Duke's argument involves a variant of Gallagher's large sieve. We then prove Theorem \ref{main thm} via a counting argument. 

\subsection{Outlook} Galois representations arise from motives and automorphic forms. The Galois representations that arise from elliptic curves are the most natural examples to consider. We expect that this methodology has the potential to lead to interesting generalizations.

\subsection{Related work} The author was informed upon the public release of an earlier version of this manuscript that Vittoria Cristante is investigating related questions in her thesis. Notably, she has concurrently obtained intriguing results connecting Malle's conjecture to the analysis of Galois representations associated with elliptic curves.

\subsection*{Acknowledgment} When the project was started, the author was a Simons postdoctoral fellow at the Centre de recherches mathematiques in Montreal, Canada. At this time, the author's research is supported by the CRM Simons postdoctoral fellowship. He would like to thank Chris Wuthrich for his comments in response to a question he posted on MathOverflow. He would like to thank Robert Lemke-Oliver for pointing out a calculation error in the previous version. Lastly, he would also like to thank the anonymous referee for the excellent report.

\section{Elliptic curves and division fields}\label{s 2}
\subsection{Parametrizing $\op{GL}_2(\F_\ell)$-extensions of $\Q$} \par Let $\ell$ be an odd prime number. In this section, we recall some well known results about the division field associated with the $\ell$-torsion of an elliptic curve $E_{/\Q}$. Throughout, $\bar{\Q}$ will be a fixed algebraic closure of $\Q$, the absolute Galois group of $\Q$ is denoted $\op{G}_\Q:=\op{Gal}(\bar{\Q}/\Q)$. The field with $\ell$ elements is denoted by $\F_\ell$. Given a finite group $G$ and a homomorphism $\rho:\op{G}_\Q\rightarrow G$, the field $\Q(\rho)$ denotes $\bar{\Q}^{\op{ker}\rho}$, and is referred to as the field \emph{cut out by $\rho$}. We note that $\Q(\rho)/\Q$ is a Galois extension, and that the Galois group $G_\rho:=\op{Gal}(\Q(\rho)/\Q)$ injects into $G$ via $\rho$, and we may identify $G_\rho$ with the image of $\rho$. In particular, if $\rho$ is surjective, then, $\Q(\rho)$ is a Galois extension of $\Q$ with Galois group $G_\rho$ isomorphic to $G$. We shall refer to such an extension as a $G$-extension of $\Q$.
\par Let $E_{/\Q}$ be an elliptic curve and $\ell$ be a prime number. Denote by $E[\ell]$ the $\ell$-torsion subgroup of $E(\bar{\Q})$, and note that $E[\ell]$ is a $2$-dimensional $\F_\ell$-vector space. Choosing an $\F_\ell$-basis for $E[\ell]$, we identify the automorphism group of $E[\ell]$ with $\op{GL}_2(\F_\ell)$. Set $\bar{\rho}_{E, \ell}:\op{G}_\Q\rightarrow \op{GL}_2(\F_\ell)$ to be the Galois representation on $E[\ell]$ and denote by $L_\ell(E)$ the subfield of $\bar{\Q}$ fixed by the kernel of $\bar{\rho}_{E, \ell}$. In other words, $L_\ell(E)$ is the field cut out by $\bar{\rho}_{E, \ell}$. Set $G_\ell(E):=\op{Gal}(L_\ell(E)/\Q)$; note that if $\bar{\rho}_{E,\ell}$ is surjective, then, $G_\ell(E)$ is isomorphic to $\op{GL}_2(\F_\ell)$.
\subsection{Elliptic curves ordered by height}\par An elliptic curve $E_{/\Q}$ admits a unique Weierstrass equation $y^2=x^3+Ax+B$, where $(A,B)\in \Z^2$ is such that $p^4\nmid A$ and $p^6\nmid B$ for all primes $p$ (cf. \cite{duke1997elliptic}). The Weierstrass equation $y^2=x^3+Ax+B$ is minimal at every prime, and the naive height of $E$ is defined as follows
\[H(E):=\op{max}\{|A|^3, B^2\}.\] Let $\cC$ be the set of isomorphism classes of elliptic curves $E_{/\Q}$ and let $\cC(X)$ be the set of elliptic curves with $H(E)\leq X^6$. We identify $\cC(X)$ with the set of tuples $(A,B)\in \Z^2$ such that $p^4\nmid A$ and $p^6\nmid B$ for all primes $p$, and $|A|\leq X^2$ and $|B|\leq X^3$. It is easy to see that
\begin{equation}\label{CS asymptotic}\# \cC(X)= \frac{4}{\zeta(10)}X^5+O(X^3), \end{equation}cf. \cite[Lemma 4.3]{brumer1992average}. Let $\cS$ be a set of isomorphism classes of elliptic curves $E_{/\Q}$. For $X>0$, set $\cS(X):=\cS\cap \cC(X)$, and let $\delta(\cS)$ to denote the limit
\[\delta(\cS)=\lim_{X\rightarrow \infty} \left(\frac{\# \cS(X)}{\#\cC(X)}\right)=\frac{\zeta(10)}{4}\lim_{X\rightarrow \infty} \left(\frac{\# \cS(X)}{X^5}\right),\] provided it exists. Denote by $\Delta(E):=-16(4A^3+27B^2)$, and note that if $H(E)\leq X^6$, then, $|\Delta(E)|\leq 496 X^6$. Let $N_E$ denote the conductor of $E$. It follows from Ogg's formula that $N_E$ divides $\Delta_E$, cf. \cite[p.389]{silverman1994advanced}. The elliptic curve $E$ is modular, and the associated modular form $f$ has level $N_E$. Set $N_E':=N_{\bar{\rho}_{E, \ell}}$; results of Carayol (cf. \cite{ribet1999lectures}) imply that $N_E'$ divides $N_E$, and hence, $N_E'\leq 496X^6$.

\begin{definition}We shall denote by $\cS_\ell(X)$ the subset of $\cC(X)$ for which $E$ 
\begin{enumerate}
    \item has semistable reduction at all primes $p$, except possibly at $p=2,3$,
    \item and $\bar{\rho}_{E, \ell}$ is surjective.
\end{enumerate} For $E\in \cS_\ell(X)$, let $\gamma_\ell(E)$ be the pair $(L_\ell(E), \bar{\rho}_{E, \ell})$.
\end{definition}

\begin{proposition}\label{M prop} Set $Y_1(X):=\left(\frac{X}{496}\right)^{\frac{1}{6}}$. For $E\in \cS_\ell\left(Y_1(X)\right)$, the pair $\gamma_\ell(E)$ belongs to $\cMY{X}$. 
\end{proposition}
\begin{proof}
    By the discussion above, $N_E'\leq |\Delta_E|\leq 496Y_1(X)^6=X$, and thus, the result follows. 
\end{proof}

Set $\Delta_E':=\Delta_{L_\ell(E)}$; we study the relationship between $\Delta(E)$ and $\Delta_E'$. Write $|\Delta_E'|$ to be a product of prime powers $\prod_p p^{m_\ell(p)}$. At each prime number $p$, let $v_p$ be the valuation normalized by setting $v_p(p)=1$.

\begin{proposition}\label{calikraus}
    Let $E_{/\Q}$ be an elliptic curve and $\ell$ be a prime number. The following assertions hold.
    \begin{enumerate}
        \item\label{calikraus1} For a prime number $p\neq \ell$ at which $E$ has semistable reduction, let $m_\ell(p):=\op{ord}_p(\Delta_E')$ defined above. Then, we have that $m_\ell(p)\leq \left(\frac{\ell-1}{\ell}\right)\#\op{GL}_2(\F_\ell)$. 
        \item\label{calikraus1b} Suppose on the other hand that $E$ has additive reduction at $p$, and $p\in \{2, 3\}$, then, $m_\ell(p)\leq (3^2-1)(3^2-3)\times 68=3264$.
        \item\label{calikraus2} We have that $m_\ell(\ell)\leq C_\ell$, where $C_\ell>0$ is a constant which depends only on $\ell$.
    \end{enumerate}
\end{proposition}
\begin{proof}
    The \cite[Theoreme 1, part (a)]{cali2002p} implies that $m_\ell(p)=dD/e$, where $d=[\Q_p(E[\ell]):\Q_p]\leq \#\op{GL}_2(\F_\ell)$ and $e$ is the index of ramification in $\Q_p(E[\ell])/\Q_p$. The number \[D=\begin{cases} 0 & \text{ if }\ell\mid v_p(j(E));\\
    (\ell-1) & \text{ if }\ell\nmid v_p(j(E)).
    \end{cases}\]
    When $\ell\nmid v_p(j(E))$, it follows from the theory of the Tate curve that $e=\ell$ (cf. \cite[Corollarie 1, (1) p.7]{cali2002p}). This proves part \eqref{calikraus1}.
    \par For part \eqref{calikraus1b}, we note that by results in \cite{cali2002p}, $D\leq 68$, and $d\leq \#\op{GL}_2(\F_3)=(3^2-3)(3^2-3)$, from which the result follows. 
    \par When $p=\ell$, the result follows from the recipe in \cite{kraus1999p}.
\end{proof}

\begin{proposition}\label{F prop} Set $Y_2(X):=\left(\frac{X}{6^{3264}\ell^{C_\ell}}\right)^{\frac{\ell}{6(\ell-1)\#\op{GL}_2(\F_\ell)}}$. For $E\in \cS_\ell\left(Y_2(X)\right)$, the pair $L_\ell(E)$ belongs to $\cFY{X}$. 
\end{proposition}

\begin{proof}
 From Proposition \ref{calikraus}, we arrive at the following estimates
\[\begin{split}  |\Delta_E'|= &2^{m_\ell(2)}3^{m_\ell(3)}\ell^{m_\ell(\ell)}\left(\prod_{p\notin \{2,3,\ell\}}p^{m_\ell(p)}\right) \\ \leq & 6^{3264}\ell^{C_\ell} \left(\prod_{p\mid \Delta_E} p\right)^{\left(\frac{\ell-1}{\ell}\right)\#\op{GL}_2(\F_\ell)} \\ \leq & 6^{3264}\ell^{C_\ell} |\Delta_E|^{\left(\frac{\ell-1}{\ell}\right)\#\op{GL}_2(\F_\ell)} \\ \leq & 6^{3264}\ell^{C_\ell}Y_2(X)^{6\left(\frac{\ell-1}{\ell}\right)\#\op{GL}_2(\F_\ell)}\leq X.
\end{split}\] This proves the result. 
\end{proof}

\section{Density results}
\par In this section, we prove density results. We give a proof of Theorem \ref{main thm} at the end of the section.  
\subsection{A lower bound for the lower density of $\cS_\ell$}
\par Let $E_{/\Q}$ be an elliptic curve. A prime number $\ell$ is said to be \emph{exceptional} if $\bar{\rho}_{E, \ell}$ is not surjective.
\begin{theorem}[Duke \cite{duke1997elliptic}]\label{dukethm}
The set of elliptic curves $E_{/\Q}$ with no exceptional primes has density $1$.
\end{theorem}
Recall from the previous section that $\cS_\ell(X)$ consists of elliptic curves $E_{/\Q}$ with $H(E)\leq X^6$, such that 
\begin{enumerate}
     \item has semistable reduction at all primes $p$, except possibly at $p=2,3$,
    \item and $\bar{\rho}_{E, \ell}$ is surjective.
\end{enumerate}
We thus write $\cS_\ell(X)$ as the intersection of sets $\cD_\ell(X)$ and $\cE_\ell(X)$, where 
\[\begin{split}
& \cD_\ell(X):=\{E\in \cC(X)\mid E\text{ has semistable reduction at all primes }p\neq 2,3\};\\
& \cE_\ell(X):=\{E\in \cC(X)\mid \bar{\rho}_{E, \ell}\text{ is surjective}\}.\\
\end{split}\]

It follows from Theorem \ref{dukethm} that $\cE_\ell(X)$ has density $1$.

\begin{proposition}\label{prop 3.2}
    Let $C\in \left(0, \frac{4}{\zeta(2)}\right)$, then, for all large enough values of $X$, 
    \[\# \cS_\ell(X)> C X^5.\] In particular, there is a constant $C_1>0$, such that for all $X>0$, \[\# \cS_\ell(X)> C_1 X^5.\]
\end{proposition}

\begin{proof}
    We make a few reductions, and prove the result as a consequence of theorems of Duke \cite{duke1997elliptic} and Sadek \cite{sadek2017counting}. It suffices to show that 
    \[\liminf_{X\rightarrow \infty} \frac{\# \cS_\ell(X)}{X^5}\geq \frac{4}{\zeta(2)}.\]
    Since $\cE_\ell$ has density $1$ and $\cS_\ell=\cD_\ell\cap \cE_\ell$, the above inequality is equivalent to showing that 
    \[\liminf_{X\rightarrow \infty} \frac{\# \cD_\ell(X)}{\# \cC(X)}\geq \frac{\zeta(10)}{\zeta(2)}.\]
   This result follows from \cite[Theorem 9.3]{sadek2017counting}.
\end{proof}
\subsection{Counting elliptic curves with prescribed division field}
\par Throughout the rest of this section, we fix a Galois extension $L/\Q$ such that $\op{Gal}(L/\Q)$ is isomorphic to $\op{GL}_2(\F_\ell)$. Let $\cD_L$ be the set of elliptic curves $E\in \cC$ such that $L_\ell(E)=L$. The following is the main result of this section. 
\begin{proposition}\label{main prop}
    Let $L/\Q$ be as above. Then, there exists a positive constant $C_2$, which only depends on $\ell$, and is independent of the extension $L/\Q$, such that for all $X>0$,
    \[\#\cD_L(X)\leq C_2 X^{9/2}\log X.\]
\end{proposition}

We provide the proof at the end of this section. Let $X>0$ and let us assume without loss of generality that there exists an elliptic curve $A\in \cD_L$. Note that since $L_\ell(A)=L$, it follows that $\bar{\rho}_{A, \ell}$ is surjective. We make a choice of $\F_\ell$-basis $\cB_A$ for $A[\ell]$. Suppose that $E\in \cD_L(X)$ is another elliptic curve, and $\cB_E$ be a choice of $\F_\ell$-basis for $E[\ell]$. Then, the residual representations give rise to isomorphisms, taken with respect to $\cB_A$ and $\cB_E$ respectively
\[\begin{split}
& \bar{\rho}_{E, \ell}: \op{Gal}(L/\Q)\xrightarrow{\sim} \op{GL}_2(\F_\ell), \\
& \bar{\rho}_{A, \ell}: \op{Gal}(L/\Q)\xrightarrow{\sim} \op{GL}_2(\F_\ell). 
\end{split}\]
Therefore, there exists $\eta\in \op{Aut}\left(\op{GL}_2(\F_\ell)\right)$, such that $\bar{\rho}_{A, \ell}=\eta\circ \bar{\rho}_{E, \ell}$. We set $\op{Inn}\left(\op{GL}_2(\F_\ell)\right)$ to denote the inner automorphisms of $\op{GL}_2(\F_\ell)$, and set \[\op{Out}\left(\op{GL}_2(\F_\ell)\right):=\op{Aut}\left(\op{GL}_2(\F_\ell)\right)/ \op{Inn}\left(\op{GL}_2(\F_\ell)\right).\] The class \[\bar{\eta}_E:=\eta \op{Inn}\left(\op{GL}_2(\F_\ell)\right)\in \op{Out}\left(\op{GL}_2(\F_\ell)\right)\] associated to $E\in \cD_L(X)$ depends only on $E$ (and is independent of the choice of basis $\cB_E$). 

\par Fixing $(A,\cB_A)$ for the entire family $\cD_L$, we decompose $\cD_L$ into a disjoint union 
\[\cD_L=\bigsqcup_{\zeta\in \op{Out}\left(\op{GL}_2(\F_\ell)\right)} \cD_{L, \zeta}, \]
where
\[\cD_{L, \zeta}:=\{E\in \cD_L\mid \eta_E=\zeta\}.\]
Let $E_1, E_2\in \cD_{L, \zeta}$, and let $\cB_{E_i}$ be a choice of $\F_\ell$-basis of $E_i[\ell]$. We find that 
\[\bar{\rho}_{A, \ell}=\eta_i\circ \bar{\rho}_{E_i, \ell},\] where $\eta_1, \eta_2\in \op{Aut}\left(\op{GL}_2(\F_\ell)\right)$. Since $\eta_{E_1}=\eta_{E_2}$, there is an inner automorphism $\tau$ of $\op{GL}_2(\F_\ell)$ such that $\eta_1=\eta_2\circ \tau$. This implies that $\bar{\rho}_{E_2, \ell}=\tau\circ \bar{\rho}_{E_1, \ell}$, i.e.,
\[\bar{\rho}_{E_1, \ell}\simeq \bar{\rho}_{E_2, \ell}.\]
\begin{definition}
    Let $E_{/\Q}$ be an elliptic curve, and let $\cD_E$ to be the set of elliptic curves $E'_{/\Q}$ for which $\bar{\rho}_{E', \ell}\simeq \bar{\rho}_{E, \ell}$.
\end{definition}

\begin{proposition}\label{outer automorphism propn}
    With respect to notation above, there is a set of elliptic curves $E_1, \dots, E_t$ with $t\leq \# \op{Out}\left(\op{GL}_2(\F_\ell)\right)$, such that $\cD_L$ is contained in the union $\bigcup_{i=1}^t \cD_{E_i}$.
\end{proposition}
\begin{proof}
    For each $\zeta\in \op{Out}\left(\op{GL}_2(\F_\ell)\right)$, such that $\cD_{L, \zeta}$ is nonempty, pick an elliptic curve $E_\zeta\in \cD_{L, \zeta}$. Then, from the discussion above, we find that $\cD_{L, \zeta}$ is contained in $\cD_{E_\zeta}$, and the result follows.
\end{proof}
Suppose that $E_1$ and $E_2$ are elliptic curves over $\Q$ such that $\bar{\rho}_{E_1, \ell}\simeq \bar{\rho}_{E_2, \ell}$. Let $N_i$ denote the conductor of $E_i$. For any prime number $p\nmid N_1N_2\ell$, note that $\bar{\rho}_{E_i, \ell}$ is unramified at $p$. Let $\sigma_p$ be a choice of Frobenius at $p$, and set $t_p(E_i):=\op{trace}\left(\bar{\rho}_{E_i, \ell}(\sigma_p)\right)$.  We find that $t_p(E_1)=t_p(E_2)$ for all prime numbers $p\nmid N_1N_2\ell$. 
\par Let $(t_1, t_2, d)\in \F_\ell^3$ be such that $d\neq 0$. Let $(E_1, E_2)$ be a pair of elliptic curves defined over $\Q$, and let $\pi_{E_1, E_2}(X, t_1, t_2, d, \ell)$ to be the number of prime numbers $p\leq X$ such that 
\begin{enumerate}
    \item $p\nmid N_{1}N_{2}$, 
    \item $p\equiv d\pmod{\ell}$, 
    \item $t_p(E_i)=t_i$ for $i=1,2$.
\end{enumerate}
Let $\pi(X, d, \ell)$ be the number of prime numbers $p\leq X$ such that $p\equiv d\pmod{\ell}$. Note that if $\bar{\rho}_{E_1, \ell}\simeq \bar{\rho}_{E_2, \ell}$, then, the relation $t_p(E_1)=t_p(E_2)$ implies that 
\[\pi_{E_1, E_2}(X, t_1, t_2, d, \ell)=0,\] unless $t_1=t_2$.

\begin{proposition}\label{sieve input propn}
    Let $(t_1, t_2, d)\in \F_\ell^3$ be such that $d\neq 0$, and $(E_1, E_2)$ be any pair of elliptic curves defined over $\Q$. Then, there are constants $\delta, C>0$, such that for all $X>0$, 
    \[\frac{1}{\# \mathcal{C}(X)^2}\sum_{(E_1, E_2)\in \mathcal{C}(X)^2}\left(\pi_{E_1, E_2}(X, t_1, t_2, d, \ell)-\delta\pi(X, d, \ell)\right)^2\leq C X.\]  
\end{proposition}
\begin{proof}
    This result follows from a variant of Gallagher's large sieve, and is similar argument to that of \cite[Theorem 2]{duke1997elliptic}. For a proof of the result, we refer to \cite[Proposition 3.2]{ray2023diophantine}.
\end{proof}

\begin{proposition}\label{prop 3.7}
    Let $E_{/\Q}$ be an elliptic curve. Then, there exists a constant $C_3>0$, which are independent of $E$, such that for all $X>0$,
    \[\# \cD_E(X)\leq C_3 X^{9/2}\log X.\]
\end{proposition}
\begin{proof}
    Let $E_1, E_2$ be elliptic curves in $\cD_E$; notice that $\bar{\rho}_{E_1,\ell}\simeq \bar{\rho}_{E_2, \ell}$. Taking $(t_1, t_2, d)\in \F_\ell^3$ such that $d\neq 0$ and $t_1\neq t_2$, we find that \[\pi_{E_1, E_2}(X, t_1, t_2, d, \ell)=0.\]
    Proposition \ref{sieve input propn} asserts that
      \[\frac{1}{\# \mathcal{C}(X)^2}\sum_{(E_1, E_2)\in \mathcal{C}(X)^2}\left(\pi_{E_1, E_2}(X, t_1, t_2, d, \ell)-\delta\pi(X, d, \ell)\right)^2\leq C X.\]  
      Restricting the sum to $\cD_E(X)^2$, we find that 
       \[\frac{1}{\# \cC(X)^2}\sum_{(E_1, E_2)\in \cD_E(X)^2}\left(\pi_{E_1, E_2}(X, t_1, t_2, d, \ell)-\delta\pi(X, d, \ell)\right)^2\leq C X.\]  
       Since \[\pi_{E_1, E_2}(X, t_1, t_2, d, \ell)=0,\] this simplifies to give us
       \[\left(\frac{\# \cD_E(X)}{\# \cC(X)}\right)^2\leq \frac{C}{\delta^2}\frac{X}{\pi(X, d, \ell)^2}, \]
       and therefore, 
       \[\left(\frac{\# \cD_E(X)}{\# \cC(X)}\right)\leq \frac{\sqrt{C}}{\delta}\frac{X^{1/2}}{\pi(X, d, \ell)}.\]
       According to Dirichlet's theorem, $\pi(X, d, \ell)\sim \frac{1}{\varphi(\ell)}\frac{X}{\log X}$, and therefore, 
       \[\left(\frac{\# \cD_E(X)}{\# \cC(X)}\right)=O\left(\frac{\log X}{\sqrt{X}}\right).\] Since $\#\cC(X)\sim \frac{4}{\zeta(10)}X^5$, we deduce that 
       \[\# \cD_E(X)=O\left(X^{9/2}\log X\right),\] which proves the result. 
\end{proof}
\begin{proof}[Proof of Proposition \ref{main prop}]
    Proposition \ref{outer automorphism propn} asserts that there is a set of elliptic curves $E_1, \dots, E_t$ with $t\leq \# \op{Out}\left(\op{GL}_2(\F_\ell)\right)$, such that $\cD_L$ is contained in the union $\bigcup_{i=1}^t \cD_{E_i}$. Proposition \ref{prop 3.7} implies that there is a constant $C_3>0$, such that  \[\# \cD_{E_i}(X)\leq C_3 X^{9/2}\log X.\] Therefore, setting $C_2:=\# \op{Out}\left(\op{GL}_2(\F_\ell)\right) C_3$, we find that 
    \[\# \cD_{L}(X)\leq C_2 X^{9/2}\log X.\]
\end{proof}
\subsection{Proof of the main result}
\par In this section, we prove the main result. The main ingredients will be Propositions \ref{M prop}, \ref{F prop}, \ref{prop 3.2} and \ref{main prop}.

\begin{proof}[Proof of Theorem \ref{main thm}]
   Setting $Y_1(X):=\left(\frac{X}{496}\right)^{\frac{1}{6}}$, we recall that Proposition \ref{M prop} asserts that for any elliptic curve $E\in \cS_\ell\left(Y_1(X)\right)$, the pair $\gamma_\ell(E)$ belongs to $\cMY{X}$.
   This implies that 
   \[\begin{split}\#\cS_\ell(Y_1(X)) & \leq \sum_{(L, \rho)\in \cMX} \# \{E\in \cC(Y_1(X))\mid \gamma_\ell(E)=(L, \rho)\};\\ &\leq  \sum_{(L, \rho)\in \cMX}\# \{E\in \cC(Y_1(X))\mid L_\ell(E)=L\}; \\ 
   & = \sum_{(L, \rho)\in \cMX}\# \cD_L(Y_1(X)) \\
   & \leq C_2 Y_1(X)^{9/2}\log Y_1(X) \#\cMX  ;\end{split}\]
   where in the final inequality, we invoke Proposition \ref{main prop}. We reiterate here that $C_2$ depends only on $\ell$, and is independent of the extension $L/\Q$. It follows from Proposition \ref{prop 3.2} that there is a constant $C_1>0$, such that for all $X>0$, we have that $\#\cS_\ell(X)>C_1X^5$. Thus, we find that 
   \[\begin{split}
   \# \cMX &\geq \frac{\#\cS_\ell(Y_1(X))}{C_2 Y_1(X)^{9/2}\log Y_1(X) };\\
   & \geq \frac{C_1(Y_1(X))^5}{C_2 Y_1(X)^{9/2}\log Y_1(X) }\geq C\frac{X^{\frac{1}{12}}}{\log X},
   \end{split}\]
   for some constant $C>0$.
   \par Setting $Y_2(X):=\left(\frac{X}{6^{3264}\ell^{C_\ell}}\right)^{\frac{\ell}{6(\ell-1)\#\op{GL}_2(\F_\ell)}}$, we recall that Proposition \ref{F prop} asserts that for any elliptic curve $E\in \cS_\ell\left(Y_2(X)\right)$, the division field $L_\ell(E)$ belongs to $\cFY{X}$. By the same argument as above, we find that 
    \[\begin{split}
   \# \cFX &\geq \frac{\#\cS_\ell(Y_2(X))}{C_2 Y_2(X)^{9/2}\log Y_2(X) };\\
   & \geq \frac{C_1(Y_2(X))^5}{C_2 Y_2(X)^{9/2}\log Y_2(X) }\geq C\frac{X^{\frac{\ell}{12(\ell-1)\#\op{GL}_2(\F_\ell)}}}{\log X},
   \end{split}\]for some constant $C>0$.
\end{proof}
\subsection*{Data availability} No data was generated or analyzed in this paper.
\bibliographystyle{alpha}
\bibliography{references}
\end{document}